\documentclass{article}
\usepackage{amsmath,amsfonts,amsthm,amssymb,array,cite,prettyref}

\DeclareFontFamily{U}{wncy}{}
\DeclareFontShape{U}{wncy}{m}{n}{<->wncyr10}{}
\DeclareSymbolFont{mcy}{U}{wncy}{m}{n}
\DeclareMathSymbol{\Sh}{\mathord}{mcy}{"58}

\newcommand{\lam}{\lambda}

\newcommand{\vep}{\varepsilon}

\newcommand{\pref}[1]{\prettyref{#1}}
\newcommand{\Cal}[1]{\mathcal{#1}}
\newcommand{\bb}[1]{\mathbb{#1}}

\newcommand{\mfrak}[1]{\mathfrak{#1}}
\newcommand{\wbar}[1]{\overline{#1}}

\newcommand{\paren}[1]{\left( #1 \right)}

\newcommand{\abs}[1]{\left| #1 \right|}

\theoremstyle{plain}
\newtheorem{thm}{Theorem}

\newtheorem{lma}{Lemma}

\newtheorem{cond}{Condition}

\theoremstyle{definition}

\theoremstyle{remark}
\newtheorem{rmk}{Remark}

\newrefformat{eq}{\textup{(\ref{#1})}}
\newrefformat{map}{\textup{(\ref{#1})}}
\newrefformat{cond}{\textup{(\ref{#1})}}

\DeclareMathOperator{\Sel}{Sel}
\DeclareMathOperator{\SelRk}{SelRk}

\title{Bounds on the Mordell--Weil rank of elliptic curves over imaginary quadratic number fields with class number 1}
\author{Erik Wallace}

\begin{document}
\maketitle

\begin{abstract}
 We generalize the lemmas of Thomas Kretschmer to arbitrary number fields, and apply them with a 2-descent argument 
 to obtain bounds for families of elliptic curves over certain imaginary quadratic number fields with class number 1.
 One such family occurs in the congruent number problem. We consider the congruent number problem over these 
 quadratic number fields, and subject to the finiteness of Sha, we show that there are infinitely many numbers 
 that are not congruent over Q but become congruent over $Q(\zeta_3)$.
\end{abstract}

\section{Introduction}

In \cite{AlvaroMax} the authors study bounds on the rank of elliptic curves
with Weierstrass equation
\[
 y^2=x^3+ax^2+bx
\]
over $\bb{Q}$, using the methods of Kretschmer \cite{Kretschmer1}.
In particular, there is a bound for the dimension of the 2-Selmer group
in terms of the number of divisors of $4b$ and $b^2-4ab$
In theorem 4.3 of \cite{AlvaroMax}, by specializing to $a=0$, the authors show that 
in order for the Selmer group to be as large as allowed by this bound, 
the primes dividing $b$ must satisfy what they call the Legendre condition, namely
\begin{cond}[Legendre Condition]
The primes $p|b$ must be congruent to $1\bmod 8$, and for any distinct pair
of primes $p,q|b$ we must have $\paren{\frac{p}{q}}=\paren{\frac{q}{p}}=1$.
\end{cond}

In the current article, we generalize Kretschmer's methods to other number fields, and apply
the new lemmas to obtain several interesting results.
While the lemmas generalizing those of Kretschmer are stated very generally, 
the theorems in this article always assume that $K$ is a quadratic imaginary field of class number 1,
for which 2 is inert, hence $K=\bb{Q}(\sqrt{D})$ where $D=-3,-11,-19,-43,-67,$ or $-163$.
The reasons for these assumptions are explained at the beginning of section~\ref{SecRankBounds}.

In contrast to the situation over $\bb{Q}$, if we start with an elliptic curve of the type
\[
 y^2=x^3+bx
\]
with $b\in \bb{Z}$, for which all of the primes dividing $b$ are inert 
then the Legendre condition can be dropped over $K$. Following \cite{AlvaroDist}, we define
the 2-Selmer rank over a number field $K$ to be
\begin{equation}\label{eq:selrk}
  \SelRk_2(E/K)=\dim_{\bb{F}_2}\Sel_2(E/K)-\dim_{\bb{F}_2}(E(K)[2]).
\end{equation}
Thus, using this terminology, we can say that it is easier to obtain families of constant 2-Selmer rank over 
a number field $K$ than it is over $\bb{Q}$, if we restrict to prime factors that are inert.
In particular we obtain the following theorem:

\begin{thm}\label{thm-Inert}
 Let $K=\bb{Q}(\sqrt{D})$ where $D=-3,-11,-19,-43,-67,$ or $-163$.
 Let $b=\pm\prod_{i}^n p_i$, where each $p_i>2$ is inert in $K$, and $p_i\neq p_j$ when $i\neq j$.
 Then for $E_b:y^2=x^3+bx$ we have
 \[
   \SelRk_2(E_b/K)=\begin{cases}
                2n+1&\text{if }b\equiv 1\bmod 8,\\
                2n&\text{if }b\equiv 3\bmod 4,\\
                2n-1&\text{if }b\equiv 5\bmod 8.
               \end{cases}
 \]
\end{thm}

On the other hand, if we allow for split-primes in the factorization of $b$, then the situation
becomes much more complicated. For just one prime, if $b=\pm p$ where $p$ splits over $K$,
and if $p=\alpha_1\alpha_2$ is a factorization, then the 2-Selmer rank can also depend on the trace of $\alpha_1$.

\begin{thm}\label{thm-split}
Let $K=\bb{Q}(\sqrt{D})$ where $D=-3,-11,-19,-43,-67,$ or $-163$, and let $p$ be a prime that splits over $K$. 
For $E_{-p}:y^2=x^3-px$ and $E_p:y^2=x^3+px$ we have
 \[
  \SelRk_2(E_{-p})=\begin{cases}
		3+\paren{\frac{t}{p}}&\text{if }p\equiv 1\bmod 8\\
                2&\text{if }p\equiv 5\bmod 8\\
                1&\text{if }p\equiv 3\bmod 4
               \end{cases}
 \quad\text{and}\quad
     \SelRk_2(E_p)=\begin{cases}
                4+\paren{\frac{t}{p}}&\text{if }p\equiv 1\bmod 8\\
                2+\paren{\frac{t}{p}}&\text{if }p\equiv 5\bmod 8\\
                2&\text{if }p\equiv 3\bmod 4
               \end{cases}
 \]
where $t$ is the trace of $\alpha$ in $K$ for a prime ideal $\alpha\Cal{O}_K$ above $p$.
\end{thm}

We do not state a general theorem for more than one split prime, but it can be expected that something
analogous to the Legendre condition must hold.

As an application of our methods we consider a generalization of the congruent number problem
to number fields. Over $\bb{Q}$, a number $n$ is congruent if it is the area of a right triangle
whose sides are all rational, which is to say that
\begin{equation}\label{eq:7}
 n=\frac{1}{2}ab\quad\text{and}\quad a^2+b^2=c^2
\end{equation}
where $n\in\bb{Z}$ and $a,b,c$ are nonzero rational numbers.
The elliptic curve associated with the congruent number problem is
\[
 E_n: y^2=x^3-n^2x,
\]
and it is known that $n$ is congruent if and only if $E_n$ has positive rank (see \cite{Koblitz}).
We extend this to other number fields $K$ by defining $n\in\Cal{O}_K$
to be congruent if \pref{eq:7} has a solution with $a,b,c\in K$, where $a$, $b$, and $c$ are nonzero. The algebraic construction relating this
problem to the curve $E_n$ still holds, and so we can say that if $E_n$ has positive rank over $K$, then $n$
is a congruent number over $K$. In this context, we can prove the following theorem.

\begin{thm}\label{thmCongNum}
 Let $K=\bb{Q}(\sqrt{D})$ where $D=-3,-11,-19,-43,-67,$ or $-163$. Let $n\in\bb{Z}^+$ be squarefree, and
 such that if $p$ is a prime dividing $n$, then $p$ is inert. Then, for $E_n:y^2=x^3-n^2x$ we have
 \[
 \SelRk_2(E_n)=\begin{cases}
                2k &\text{if $n$ is odd}\\
                2k-1&\text{if $n$ is even}
               \end{cases}
 \]
 where $k$ is the number of prime factors of $n$. In particular, 
 if $\Sh(E_n/K)[2^\infty]$ is finite for all such $n$, then every even squarefree number $n$
 relatively prime to $D$ and divisible only by primes that are inert in $K$ is a congruent number over $K$.
\end{thm}

Several results are available  providing criteria for certain even square-free numbers are {\it not} congruent numbers,
including the following:
\medskip

\begin{tabular}{ll}
 Genocchi (1855)&$n=2p$ or $n=2pq$ such that $p$ and $q\equiv 5\bmod 8$ and $p\neq q$,\\
 Bastien (1913)&$n=2p$ such that $p\equiv 9\bmod 16$,\\
 Lagrange (1974)&$n=2pq$ such that $p\equiv 1\bmod 8$, $q\equiv 5\bmod 8$, and $\left(\frac{p}{q}\right)=-1$.
\end{tabular}
\medskip

Then under the assumption that $\Sh(E_n/K)[2^\infty]$ is finite, this theorem implies the existence of
infinitely many numbers $n$ that are not congruent over $\bb{Q}$ but become congruent over $K$.
For example if $K=\bb{Q}(\zeta_3)$ and $p\equiv 41\bmod 48$ or $p\equiv 5\bmod 24$, then $2p$ is not congruent over $\bb{Q}$
but would be congruent over $\bb{Q}(\zeta_3)$. Since the conditions for a prime $p$ to be 
inert over a quadratic number field can be completely described by congruences, then by the 
Chinese Remainder Theorem, and Dirichlet's theorem on primes in arithmetic progressions,
it would follow that there are infinitely many $n$ that are not congruent numbers over $\bb{Q}$,
but which become congruent over all of the fields in theorem~\ref{thmCongNum}.

In section~\ref{secLocal}, we present some basic results on square in local fields that are required to
prove the generalization of Kretschmer's lemmas to higher degree number fields, as carried out in section~\ref{secKretschmer}. 
Then in section~\ref{SecRankBounds} we prove bounds on the rank of certain elliptic curves over the number fields $K=\bb{Q}(\sqrt{D})$
where $D=-3,-11,-19,-43,-67,$ or $-163$ using descent by 2-isogeny. The lemmas in section~\ref{secKretschmer} are used
to prove the local solvability of the homogeneous spaces corresponding to an elliptic curve $E$ and its 2-isogenous curve $E'$.


\section*{Acknowledgements}
 The author would like to thank Prof. Dr. Kretschmer for scanning his Diplomarbeit, which proved to be very
helpful for generalizing his methods. The author would also like to thank \'Alvaro Lozano-Robledo and
Keith Conrad for useful discussions and feedback.

\section{Squares in Local fields}\label{secLocal}
For now let $K$ be an arbitrary number field.
Let $A_\nu$ be the ring of integers of the completion $K_\nu$ with respect
to a non-archimedian valuation $\nu$, let $\mfrak{p}$ be the maximal ideal of $A_\nu$,
and let $\pi$ be a uniformizer. Then, for all $a\in K_\nu^\times$ we have a unique representation
\[
 a=\pi^{\nu(a)}\vep
\]
where $\vep\in A_\nu^\times$. Now if $b\in (K_\nu^\times)^2$, it means that there exists $a\in K_\nu^\times$
such that 
\[
 b=a^2=\pi^{2\nu(a)}\vep^2,
\]
thus $b\in (K_\nu^\times)^2$ if and only if both of the following conditions hold:
\begin{enumerate}
 \item\label{cond:1} $\nu(b)\equiv 0\mod 2$,
 \item\label{cond:2} $b/\pi^{\nu(b)}\in (A_\nu^\times)^2$.
\end{enumerate}
It remains to determine the description of squares in $A_\nu^\times$. Let $p$ be the characteristic
of the residue field $k_\nu$. If $p$ is odd, then by Hensel's Lemma, $x^2-\vep$ has a solution
in $K_\nu^\times$ if an only if $x^2-\vep$ has a solution in $k_\nu^\times$. In the following,
we will use $\chi$ to denote the quadratic character on $K_\nu^\times$ induced by the natural map to $k_\nu^\times/(k_\nu^\times)^2$.
That is for $\alpha\in K$ such that $\nu(\alpha)=0$ we have
\[
 \chi(\alpha)=\begin{cases}
          1&\text{if $x^2-\alpha \bmod \mfrak{p}$ has a solution,}\\
          -1&\text{otherwise.}
         \end{cases}
\]
For $p=2$, Hensel's Lemma gives us
\[
 \chi(\alpha)=\begin{cases}
          1&\text{if $x^2-\alpha \bmod \mfrak{p}^{2e+1}$ has a solution,}\\
          -1&\text{otherwise.}
         \end{cases}
\]
where $e$ is the ramification index. In particular if $K$ is a quadratic number field for which 2 is inert,
then $K_\nu$ is an unramified quadratic extension of $\bb{Q}_2$, thus $K_\nu$ is generated by a root of $x^2+x+1$,
which will still be called $\zeta_3$. It is then possible to consider elements of $K_\nu$ using the basis $1,\zeta_3$ 
and thus we can say that $\alpha\in A_\nu^\times$ is a square if and only if
\begin{equation}\label{eq:3}
 \alpha\equiv 1,\zeta_3,\zeta_3^2,5,5\zeta_3,5\zeta_3^2\bmod 8A_\nu.
\end{equation}

Equivalently, we can say that $\alpha\in A_\nu^\times$ is a
square if and only if $\alpha^3$ reduces to 1 or 5 mod $8A_\nu$, and indeed if $\alpha\in\bb{Z}_2^\times$, then $\alpha$ is a 
square in $\alpha\in A_\nu^\times$ if and only if $\alpha^3$ reduces to $1\bmod 4A_\nu$.
Furthermore, we see that $\alpha$ is a fourth power if and only if it reduces to
\[
 1,\zeta_3,\zeta_3^2 \bmod(8),
\]
or equivalently if $\alpha^3$ reduces to $1\bmod 8A_\nu$.
Note that $-1$ is not a square in $K_\nu$ since $K_\nu$ is an unramified extension of $\bb{Q}_2$,
whereas $\bb{Q}_2(\sqrt{-1})$ is a ramified extension.

We now prove several lemmas on character sums, which can
be combined with Hensel's Lemma to show that $f(x)$ is a non-zero square in
$A^\times$ for some $x$, where $f(x)=cx^2+d$ or $f(x)=cx^4+d$.

\begin{lma}\label{lma1}
 Let $K$ be a number field, let $\nu$ be a non-archimedian valuation such that $\nu(cd)=0$ 
 and $|k_\nu|=q>3$, and let $\chi$ be the quadratic character of $K_\nu^\times$ defined as above. 
 Then, there exists $x\in \bb{F}_q$ such that $\chi(cx^2+d)=1$.
\end{lma}

Note that the lemma does not assume that $K$ is quadratic.

\begin{proof}
 The character $\chi$ has order 2, and $cx^2+d$ has 2 distinct roots by the assumptions in the lemma, then
as a special case of Exercise~5.58 in \cite{Lidl} we have
\[
 \sum_{x\in\bb{F}_q} \chi(cx^2+d)=-\chi(c)=\pm 1.
\]
On the other hand if we assume that $\chi(cx^2+d)$ is never equal to 1, then 0 can occur at most twice, so we have
\[
 \sum_{x\in\bb{F}_q} \chi(cx^2+d)\leq 2-q< -1,
\]
since $q> 3$, giving a contradiction.
\end{proof}

If 3 is totally ramified or split in $K$, then this lemma does not apply. But in those cases, the
residue field is isomorphic to $\bb{Z}/3\bb{Z}$ and we can still take $0$ and $\pm 1$ as representatives.
By plugging directly into $f(x)=\pm x^2\pm 1$ we see that $\chi(x^2-1)\neq1$ regardless
of the choice of $x$, but for the other three options we can find $x$ such that $\chi(f(x))=1$.
We therefore have the following replacement of Lemma~\ref{lma1}.
\begin{lma}\label{lma1a}
Let $K$ be a number field in which 3 is split or totally ramified, let $\nu$ be a valuation extending
the 3-adic valuation, let $\chi$ be the quadratic character as defined above,
and suppose $\nu(cd)=0$. Then, $\chi(cx^2+d)=1$ for some $x\in \bb{F}_3$ if and only if $\nu(d-1)=0$ or $\nu(c+1)=0$.
\end{lma} 

Lemma~\ref{lma1a} remains valid for $f(x)=cx^4+d$, and Lemma~\ref{lma4} below holds for $|k_\nu|=q>7$, but $q=5$ is problematic.

\begin{lma}\label{lma4}
Let $K$ be a number field, let $\nu$ be a non-archimedian valuation such that $\nu(cd)=0$ 
 and $|k_\nu|=q>7$, and  let $\chi$ be the quadratic character of $K^\times$ defined as above. 
 Then, there exists $x\in \bb{F}_q$ such that $\chi(cx^4+d)=1$. 
\end{lma}
\begin{proof}
If $q\equiv 3\bmod 4$, then the proof of Lemma~\ref{lma1} goes through unchanged.
If $q\equiv 1\bmod 4$, then Exercise 5.58 of \cite{Lidl} gives us
\[
 \sum_{x\in\bb{F}_q} \chi(cx^4+d)=-\chi(c)+2\Re(\wbar{\lam}(c)\lam(-d)J(\lam,\chi)),
\]
where $\lam$ is a degree 4 character, and then by Theorem~5.22 of \cite{Lidl} we obtain
\[
 \abs{\chi(c)+\sum_{x\in\bb{F}_q} \chi(cx^4+d)}\leq 2\sqrt{q}.
\]
On the other hand, if we assume that $\chi(cx^4+d)=1$ never occurs, then
\[
 \sum_{x\in\bb{F}_q} \chi(cx^4+d)\leq 2-q.
\]
These inequalities give us a contradiction when $q>9$. In the case of $q=7$, a quick calculation with sage or magma shows that the result still holds.
\end{proof}
\begin{rmk}
It can be checked with sage or magma that for $q=5$ there are exceptions when $(c,d)$ is
\[
 (1,2),\quad(2,3),\quad(3,2),\quad\text{or}\quad(4,3).
\]
\end{rmk}


\section{Generalizing Kretschmer's lemmas}\label{secKretschmer}
Consider the elliptic curve
\begin{equation}\label{eq:1}
 E:y^2=x(x^2+ax+b)
\end{equation}
where $a,b\in\Cal{O}_K$ and suppose that $b$ factors as $b=b_1b_2$ in $\Cal{O}_K$. Then, the equations
\[
x=b_1\frac{u^2}{w^2}\quad y=b_1\frac{uv}{w^3}
\]
define a $K$-rational map from
\begin{equation}\label{eq:2}
 C:v^2=g(u,w)=b_1u^4+au^2w^2+b_2w^4
\end{equation}
to $E$, and thus there is also a map from $C$ to any elliptic curve isogenous to $E$, in particular to the curve
\[
 E':y^2=x^3-2ax^2+(a^2-4b)x.
\]
We will consider when $C$ has solutions over various localizations of $K$. When $K$ is an
imaginary quadratic field, there is only one infinite place with completion $\bb{C}$, and $C$
always has complex solutions. For the completions at any finite place we have $C(K_\nu)\neq \emptyset$
if and only if $v^2=g(u,1)$ or $v^2=g(1,w)$ has a solution in $A_\nu$,
where we consider only triples $u,v,w$ yielding a valid projective point.
Kretschmer's lemmas 3.25 and 3.27 in \cite{Kretschmer1} were proved for $p$-adic completions of $\bb{Q}$, 
where $p$. He did not provide a lemma for $p=2$ or 3, nor did he provide a lemma for the case where $a=0$.
Alternately see lemmas 2 and 3 in \cite{Kretschmer2}. The first two lemmas in this section are direct generalizations
of his lemmas 3.25 and 3.27 to finite extensions of $\bb{Q}_p$. However, we also prove lemmas for $p=2$ and $p=3$, and for the case
where $a=0$. Since the strongest result for $p=2$ was proved with the help of Sage,
an analogous lemma would have been very hard to prove at the time that Kretschmer first obtained
his results.


%

%
The following two lemmas are generalizations of lemmas 3.25 and 3.27 in \cite{Kretschmer1}, and are not
restricted to imaginary quadratic fields.

\begin{lma}\label{lma3.25}
 Let $K$ be a number field, and suppose that $|k_\nu|=q>3$ and $\mu=\nu(a^2-4b)> 0$ but $\nu(2b)=0$.
 Then $C(K_\nu)\neq \emptyset$ if and only if one of the following holds:
 \begin{enumerate}
  \item $\chi(b_1)=1$ or $\chi(b_2)=1$
  \item $\mu$ is even and $\chi(a)=\begin{cases}
                                    1 &\text{if $p\equiv 5,7\bmod 8$ and $p$ splits,}\\
                                    -1&\text{otherwise.}
                                   \end{cases}$
 \end{enumerate}
\end{lma}
\begin{rmk}
  Naturally, for a given quadratic field, the conditions in the lemma can be expressed completely 
  by congruence conditions, by determining the congruences for the splitting of $p$.
\end{rmk}
\begin{proof}
 First suppose $C(K_\nu)\neq \emptyset$, and let $a^2-4b=\pi^\mu\delta$ where $\delta\in A_\nu$.
 By taking $w=1$ and completing the square, \pref{eq:2} can be brought into the form
\begin{equation}\label{eq:4}
 (2b_1u^2+a)^2=a^2-4b+4b_1v^2.
\end{equation}
It follows that $\pi^\mu\delta+4b_1v^2$ is a square in $A_\nu$.

If $\nu(v)=0$, then $\nu(\pi^\mu\delta+4b_1v^2)=0$, since $\mu>0$, hence $\pi^\mu\delta+4b_1v^2\in A_\nu^\times$.
By reducing mod $\mfrak{p}$ it follows immediately that $\chi(b_1)=1$.

If $\nu(v)>0$ and $\mu$ is odd, then we must have $\mu>2\nu(v)$ otherwise we have a contradiction with 
$\pi^\mu\delta+4b_1v^2$ being a square. On the other hand $\mu<2\nu(v)$ implies that $\chi(b_1)=1$.

If $\nu(v)>0$ and $\mu$ is even, then $\nu(\pi^\mu\delta+4b_1v^2)>0$, hence $\nu(2b_1u^2+a)>0$ which
means that $-2ab_1$ is a square $\bmod\,\mfrak{p}$. Now if $\chi(b_1)= -1$, then
\[
\chi(-2ab_1)=-\chi(-1)\chi(2)\chi(a)=1 
\]
If $p$ is inert, then all integers not divisible by $p$ have a non-zero square root in $k_\nu$.
In particular this is true of $-1$ and $2$, so it follows that $\chi(a)=-1$. If $p$ splits, then
\[
 \chi(-1)=\paren{\frac{-1}{p}}=(-1)^{\frac{p-1}{2}}
 \text{ and }
 \chi(2)=\paren{\frac{2}{p}}=(-1)^{\frac{p^2-1}{8}}
\]
giving the result stated in the lemma.

Conversely if $\chi(b_1)=1$, simply take $u=1$ and $w=0$ and take $v$ to be the square root of $b_1$.
The situation for $\chi(b_2)=1$ is similar with $u$ and $w$ reversed. Henceforth, suppose
$\chi(b_1)=\chi(b_2)=-1$,
\[
\chi(a)=\begin{cases}
                                    1 &\text{if $p\equiv 5,7\bmod 8$ and $p$ splits},\\
                                    -1&\text{otherwise.}
                                   \end{cases},
\]
and also that $\mu$ is even. By Lemma~\ref{lma1} and Hensel's lemma,
there exists $v_0\in A_\nu$ such that $4b_1v_0^2+\delta$ is a square in $A_\nu^\times$.
Let $\mu=2m$ and $v=\pi^mv_0$. Then,
\[
 a^2-4b+4b_1y^2=p^{2m}(\delta+4b_1y_0^2),
\]
thus by Hensel's lemma there exists $\rho\in A_\nu$ such that
\[
 \rho^2=p^{2m}(\delta+4b_1y_0^2).
\]
Under the assumptions on $\chi(b_1)$ and $\chi(a)$ we have $\chi(-2b_1a)=1$, thus by
Hensel's lemma again, there exists $u\in A_\nu$ such that
\[
 u^2=-\frac{1}{2b_1}(a-\rho)
\]
since $\rho$ vanishes mod $\mfrak{p}$. Therefore
\[
 (2b_1u^2+a)^2=\rho^2=a^2-4b+4b_1v^2.
\]
\end{proof}

\begin{lma}\label{lma3.27}
 Let $K$ be a number field, and suppose that $|k_\nu|=q>3$ and $\nu(b)>0$ but $\nu(2a)=0$.
 Then $C(K_\nu)\neq\emptyset$ if and only if one of the following holds:
 \begin{enumerate}
  \item\label{cond:i} $\nu(b_1+b_2)=0$
  \item\label{cond:ii} $\nu(b_1+b_2)>0$ and $\chi(a)=1$ or $\nu(b_1)$ or $\nu(b_2)$ is even.
  \end{enumerate}
\end{lma}

\begin{proof}
 If $\nu(b_1+b_2)=0$, then by symmetry between $b_1$ and $b_2$ we may suppose that
 $\nu(b_1)>0$ and $\nu(b_2)=0$. Then by Lemma~\ref{lma1} there exists $u\in A_\nu^\times$
such that $\chi(au^2+b_2)=1$. Since $\nu(b_1)=1$, then by Hensel's lemma, 
there exists a solution to \pref{eq:2} with $w=1$.
 
 If $\nu(b_1+b_2)>0$, then $\nu(b_1)>0$ and $\nu(b_2)>0$, since $\nu(b_1b_2)=\nu(b)>0$. If $\chi(a)=1$,
 then simply take $u=w=1$ and apply Hensel's lemma to obtain $v$. If $\chi(a)=-1$, then by
 symmetry between $b_1$ and $b_2$ we may suppose that $\nu(b_2)$ is even. Let $b_2=\pi^{2n}\beta_2$,
 where $n\in\bb{Z}^+$. Then by Lemma~\ref{lma1} and Hensel's lemma there exists $u_0\in A_\nu$ 
 such that $\chi(au_0^2+\beta_2)=1$. Let
 $u=\pi^nx_0$ and let $b_1=\pi^{\nu(b_1)}\beta_1$ with $\beta_1\in A_\nu$, and let $w=1$. Then,
 \[
  b_1u^4+au^2w^2+b_2w^4=\pi^{4n+\nu(b_1)}\beta_1u_0^4+\pi^{2n}(au_0^2+\beta_2).
 \]
Since $4n+\nu(b_1)>2n$, then by Hensel's lemma there exists $v\in A_\nu$ such that
\[
 v^2=b_1u^4+au^2w^2+b_2w^4.
\]

Conversely suppose that $C(k_\nu)\neq\emptyset$. Then, $v^2=g(u,1)$ or $v^2=g(1,w)$ has a solution in $A_\nu$.
We consider the first, case, since the other is completely analogous. Let $u=\pi^nu_0$ satisfy
$v^2=g(u,1)$, where $\nu(u)=n$ and $u_0\in A_\nu^\times$, and let $b_i=\pi^{\nu(b_i)}\beta_i$.
Then
\[
 g(u,1)=\pi^{4n+\nu(b_1)}\beta_1u_0^4+\pi^{2n}au_0^2+\pi^{\nu(b_2)}\beta_2
\]
If \ref{cond:i} does not hold, then $\nu(b_1+b_2)>0$.
Furthermore, if $\chi(a)=1$ does not hold, then $\chi(a)=-1$. We will show that $\nu(b_2)$ is even.
Suppose $\nu(b_2)$ is odd. If $2n<\nu(b_2)$, then $\nu(g(u,1))=2n$ so $\chi(au_0^2)=1$, which contradicts $\chi(a)=-1$.
If $2n>\nu(b_2)$, then $\nu(g(u,1))=\nu(b_2)\equiv 1\mod2$ contradicting the fact that $g(u,1)$ is a square.
\end{proof}

Lemmas~\ref{lma3.25}~and~\ref{lma3.27} do not apply when $\nu(2)>0$ or when $|k_\nu|=3$.
We first consider the case where $|k_\nu|=3$, since it is easier to deal with.

In the proof of Lemma~\ref{lma3.25} we apply Lemma~\ref{lma1} to $\chi(4b_1v_0^2+\delta)$
where $a^2-4b=\pi^{2m}\delta$. Using Lemma~\ref{lma1a} instead requires $\nu(\delta -1)=0$ or $\nu(b_1+1)=0$.
We also know that in this case $\chi(-1)=\chi(2)=-1$, so $\chi(a)=-1$ must occur when $\chi(b_1)=\chi(b_2)=-1$.

In the proof of Lemma~\ref{lma3.27} we apply Lemma~\ref{lma1} to $\chi(au_0^2+\beta_2)$
where $b_2=\pi^{2n}\beta_2$. Using Lemma~\ref{lma1a} instead requires $\nu(\beta_2-1)=0$ or $\nu(a+1)=0$.

Since the proofs otherwise remain unchanged, we can state the following replacement of those lemmas:

\begin{lma}\label{lma3adic}
 Let $K$ be a number field such that $3$ either splits completely or is totally ramified. 
Let $\nu$ be a valuation on $K$ extending the 3-adic valuation such that $|k_\nu|=3$,
 and let $\pi$ be a uniformizer. 
 
 Suppose $\mu=\nu(a^2-4b)> 0$ but $\nu(b)=0$. Then, $C(K_\nu)\neq \emptyset$ if and only if one of the
 following hold:
 \begin{enumerate}
  \item $\chi(b_1)=1$ or $\chi(b_2)=1$
  \item $\mu$ is even, $\chi(a)=-1$, and $\nu(b_1+1)=0$ or $\nu(b_2+1)=0$ or $a^2-4b=\pi^{2m}\delta$ and $\nu(\delta-1)=0$.
 \end{enumerate}
 Suppose $\nu(b)>0$ but $\nu(a)=0$. Then, $C(K_\nu)\neq\emptyset$ if and only if one of the following hold:
 \begin{enumerate}
  \item $\nu(b_1+b_2)=0$
  \item $\nu(b_1+b_2)>0$ and $\chi(a)=1$, or $b_i=\pi^{2n}\beta_i$ where $\beta_i\in A_\nu^\times$ and $\nu(a+1)>0$ or $\nu(\beta_i-1)>0$
                         for $i=1$ or 2.
  \end{enumerate}
 \end{lma}
 
 If $\nu$ extends the 2-adic valuation, then we cannot complete the square mod $\mfrak{p}$ and so \pref{eq:4}
 cannot be used in the proofs. Moreover, while it is reasonable to prove if and only if statements when $a=0$,
 it becomes much more difficult when $a\neq 0$ because there are many more cases to check.  Since the lemmas
 considered so far do not apply when $a=0$, we include all of the lemmas applying to the $a=0$ together, and
 we prove a lemma for the 2-adic case when $a\neq 0$ now.

  \begin{lma}\label{lma2adic-nonzero}
  Let $K$ be a quadratic number field such that $2$ is inert, let $\nu$ denote the valuation extending the 2-adic valuation,
  and take $\pi=2$ to be the uniformizer. Let $a=2^{\nu(a)}\alpha$, $b_1=2^{\nu(b_1)}\beta_1$, and 
  $b_2=2^{\nu(b_2)}\beta_2$ where $\alpha,\beta_1,\beta_2\in K_\nu^\times$ and $\nu(a)>0$.
  Then $C(K_\nu)\neq\emptyset$ if one of the following conditions holds.
  \begin{enumerate}
   \item $\nu(b_1)$ or $\nu(b_2)$ is even and for that $b_i$ the corresponding $\beta_i^3$ reduces to 1 or $5\bmod 8A_\nu$.
         In particular if $b_1,b_2\in\bb{Z}_2$, then $\beta_i\equiv 1\bmod 4A_\nu$ is sufficient.
   \item $\min\{\nu(b_1)-\nu(a),\nu(b_2)-\nu(a)\}\geq 3$,
         $\nu(a)$ even, and $\alpha^3$ reduces to 1 or 5 $\bmod\, 8A_\nu$.
         In particular if $a\in\bb{Z}_2$, then $\alpha\equiv 1\bmod 4A_\nu$ is sufficient.
   \item $\nu(b_1)+\nu(b_2)-2\nu(a)\geq 3$, $\nu(a)$ even, $\nu(b_1)$ or $\nu(b_2)$ is even,
         and for the $b_i$ for which $\nu(b_i)$ is even the corresponding $\beta_i$ has the property that
         $(\alpha+\beta_i)^3$ reduces to 1 or 5 $\bmod\,8A_\nu$.
    \end{enumerate}
 \end{lma}
 \begin{rmk}
  If $K$ is a number field for which $2$ splits, then Lemma 3.3 in \cite{AlvaroMax}
  applies. If $K$ is a quadratic number field for which $2$ is ramified, then the considerations must be carried out in $\bb{Z}/32\bb{Z}$.
 \end{rmk}
  \begin{proof}
  The first condition is equivalent to $b_1$ or $b_2$ being a square, see \pref{eq:3} above, in which case, we take $(u,w)=(1,0)$ or $(0,1)$ 
  respectively. For the second case, take $u=w=1$.  For the third case, by symmetry between the $b_i$'s, 
  we may suppose that $\nu(b_1)$ is even. Let $\nu(b_1)-\nu(b_2)=2\mu$.
  Then take $u=1$ and $w=2^\mu$. It follows that
  \[
   b_1u^4+au^2v^2+b_2w^4=2^{\nu(b_1)}\beta_1+\alpha\pi^{\nu(a)+2\mu}+\beta_2 2^{\nu(b_2)+4\mu}
                        =2^{\nu(b_1)}\left(\beta_1+\alpha+\beta_2 2^{\nu(b_1)+\nu(b_2)-2\nu(a)}\right).
  \]
  Using $\zeta_3$ as a generator of $K_\nu$ over $\bb{Q}_2$, we observe that although $1+4\zeta_3$ and $5+4\zeta_3$ reduce to
  $1 \bmod 4A_\nu$, they are not in $\bb{Z}_2$. Hence, by restricting to $\bb{Z}_2$, $1\bmod 4A_\nu$ lifts to 1 or $5\bmod 8A_\nu$,
  upon which the cube can be omitted as well.
 \end{proof}
 
Consider the case where $a=0$.
 Then equation \pref{eq:2} reduces to 
 \begin{equation}\label{eq:5}
  C:v^2=b_1u^4+b_2w^4
 \end{equation}
where $b=b_1b_2$.  We also observe that 3 divides 
the conductor of $E$ if and only if $3\mid b$. Thus the case $p=3$ can be removed if we assume $3\nmid b$. Since $a=0$,
the assumptions of lemmas~\ref{lma3.25}~and~\ref{lma3.27} cannot be met, but similar methods can still be applied to equation \pref{eq:5} directly.

We now use this to prove the following replacement of lemmas \ref{lma3.25}, \ref{lma3.27}, and \ref{lma3adic}. 
\begin{lma}\label{lmapadic}
Let $K$ be a number field, and let $\nu$ be a non-archimedian valuation on $K$ such that $|k_\nu|=q$ is odd.
Suppose also that $a=0$ and $\nu(b)>0$. Let $b_1=\pi^{\nu(b_1)}\beta_1$ and $b_2=\pi^{\nu(b_2)}\beta_2$ where $\beta_1,\beta_2\in K_\nu^\times$.
Then $C(K_\nu)\neq\emptyset$ if and only if one of the following holds
\begin{enumerate}
 \item $\nu(b_1)$ is even and $\chi(\beta_1)=1$ or $\nu(b_2)$ is even and $\chi(\beta_2)=1$. In particular,
       if $p$ is inert and $b_1,b_2\in\bb{Z}_p^\times$, then $\chi(\beta_1)=\chi(\beta_2)=1$ holds automatically.
       
 \item Both $\nu(b_1)$ and $\nu(b_2)$ are even and $\nu(b_1)\equiv\nu(b_2)\bmod 4$.
 
 \item Both $\nu(b_1)$ and $\nu(b_2)$ are odd, $\nu(b_1)\equiv\nu(b_2)\bmod 4$ , and $\chi_4(-\beta_2/\beta_1)=1$. In particular
      \begin{enumerate}
       \item if $p$ splits or is totally ramified and $p\equiv 3\bmod 4$, then $\chi_4(-\beta_2/\beta_1)=1$ if and only if $\chi(\beta_1)=-\chi(\beta_2)$,
       \item if $p$ is inert, $p\equiv 1\bmod 4$, and $b_1,b_2\in \bb{Z}_{p}$, then $\chi_4(-\beta_2/\beta_1)=1$ if and only if $\chi(\beta_1)=\chi(\beta_2)$,
       \item if $p$ is inert, $p\equiv 3\bmod 4$, and $b_1,b_2\in \bb{Z}_{p}$, then $\chi_4(-\beta_2/\beta_1)=1$ holds automatically.
      \end{enumerate}
\end{enumerate}
\end{lma}

\begin{proof}
If $b_1$ or $b_2$ is a square, simply take $(u,w)=(1,0)$ or $(u,w)=(0,1)$ respectively.
Note that $b_i$ is a square if and only if $\nu(b_i)$ is even and $\chi(\beta_i)=1$.
If $p$ splits, and $b_i\in \bb{Z}_p^\times$, then $b_i$ always has square root in $A_\nu^\times$.

Suppose now that neither $b_1$ nor $b_2$ is a square.
Let $u=\pi^{\nu(u)}\vep, w=\pi^{\nu(w)}\delta$ where $\vep,\delta\in A_\nu^\times$. We thus have
\[
  \nu(b_1u^4+b_2w^4)\geq\min(\nu(b_1)+4\nu(u),\nu(b_2)+4\nu(w)).
\]
If $C(K_\nu)\neq\emptyset$, then $b_1u^4+b_2w^4$ must be a square in $K_\nu$, 
and $\nu(b_1)+4\nu(u)\neq\nu(b_2)+4\nu(w)$ implies that $b_1$ or $b_2$ must be a square, giving a contradiction.
So equality must hold, meaning that $\nu(b_1)\equiv\nu(b_2)\bmod 4$.
Additionally, if $\nu(b_1)$ is odd, then $\beta_1\vep^4+\beta_2$ must have positive valuation,
hence $\beta_1\vep^4+\beta_2\equiv 0\bmod \mfrak{p}$, which is equivalent to $\chi_4(-\beta_2/\beta_1)=1$.

Conversely if $\nu(b_1)\equiv\nu(b_2)\bmod 4$, then take $u=\vep$ and $w=\pi^\mu$ where 
$4\mu=\nu(b_1)-\nu(b_2)$. Then we have 
\begin{equation}\label{eq:6}
 b_1u^4+b_2w^4=\pi^{\nu(b_1)}(\beta_1\vep^4+\beta_2).
\end{equation}
If $\nu(b_1)$ is even and $\nu(6)=0$, then we obtain $C(K_\nu)\neq\emptyset$ by applying Lemma~\ref{lma4}.
In the case where $q=3$, it is easy to see that Lemma~\ref{lma1a} remains valid if we replace
$x^2$ by $x^4$, but the exceptional case of $\nu(\beta_1-1)>0$ and $\nu(\beta_2+1)>0$ can be dealt with
by reversing $b_1$ and $b_2$ if necessary. This trick does not work for $q=5$, because $(2,3)$ and $(3,2)$
are both among the exceptions listed at the end of the proof of Lemma~\ref{lma4}, however these cases can 
be dealt with as follows. In both cases $\beta_1+\beta_2\equiv 0\bmod \pi$, so let $\beta_1+\beta_2=\pi^n\vep$,
where $\vep\in A_\nu^\times$. If $a\in A_\nu$, then by Hensel's lemma there exists $x\in A_\nu$
such that $x^4=1+\pi^na$, thus
\[
\beta_1x^4+\beta_2=\beta_1(1+\pi^na)+\beta_2=\pi^n(\vep +\beta_1a).
\]
If $n$ is even, then take $a=\beta_1^{-1}(1-\vep)$. If $n$ is odd then take $a=\beta_1^{-1}(\pi-\vep)$.

If $\nu(b_1)$ is odd we note that $\beta_1x^4+\beta_2-\pi$ reduces to $\beta_1x^4+\beta_2\bmod \mfrak{p}$.
So if $\beta_1x^4+\beta_2\equiv 0\bmod \mfrak{p}$ has a solution, then we can lift by Hensel's lemma to
a solution of $\beta_1x^4+\beta_2-\pi=0$ in $A_\nu$, and the right hand side of $\pref{eq:6}$ is $\pi^{\nu(b_1)+1}$,
which is clearly a square. It remains to show that the condition $\chi_4(-\beta_2/\beta_1)=1$ 
can be simplified in the special cases listed.

If $p$ splits or is totally ramified and $p\equiv 3\bmod 4$, then $-1$ is not a square,
and every square in $\bb{F}_q^\times$ is also a fourth power. 
Since $-1$ is not a square and $\vep^4$ ranges over all squares as $\vep$ is varied, then
there exists $\vep$ satisfying $\beta_1\vep^4+\beta_2\equiv 0\bmod\mfrak{p}$ if and
only if $\chi(\beta_1)=\chi(-\beta_2)=-\chi(\beta_2)$.

If $p$ is inert and $b_1,b_2\in\bb{Z}_p$, then $f(x)=x^4+\beta_2/\beta_1\in\bb{F}_p[x]$.
The condition $\chi_4(-\beta_2/\beta_1)=1$ is equivalent to $f(x)$ splitting
over $\bb{F}_{p^2}$. Note that if $\alpha$ is a root of $f(x)$, then so is $-\alpha$.
This observation leads to the consideration of two cases:
\begin{enumerate}
 \item $f(x)=(x^2-c)(x^2-d)$ over $\bb{F}_p$, and
 \item $f(x)=(x^2-cx+d)(x^2+cx+d)$ over $\bb{F}_p$.
\end{enumerate}
Only the first case needs to be considered if $f(x)$ has a root in $\bb{F}_p$, but if $f(x)$
has no roots in $\bb{F}_p$, it is possible that
\[
 f(x)=(x-\alpha_1)(x+\alpha_1)(x-\alpha_2)(x+\alpha_2)
\]
and neither $\alpha_1^2$ nor $\alpha_2^2$ is in $\bb{F}_p$, which is covered by the second case.
In the first case we have
\[
 f(x)=x^4-(c+d)x^2+cd,
\]
so $d=-c$ and $\beta_2/\beta_1=cd=-d^2$, hence $\chi_4(-\beta_2/\beta_1)=1$
is equivalent to $-\beta_2/\beta_1$ being a square. In the second case we have
\[
 f(x)=x^4+(2d-c^2)x^2+d^2,
\]
so $\chi_4(-\beta_2/\beta_1)=1$ is equivalent to $\beta_2/\beta_1=d^2$ and $2d=c^2$
for some $c,d\in\bb{F}_p$. If in addition $p\equiv 1\bmod 4$, 
then $-\beta_2/\beta_1$ and $\beta_2/\beta_1$ are either both squares or both non-squares, 
hence $\chi_4(-\beta_2/\beta_1)=1$ is equivalent to $\chi(\beta_2)=\chi(\beta_1)$.
On the other hand if $p\equiv 3\bmod 4$, then one of $-\beta_2/\beta_1$ and $\beta_2/\beta_1$ 
is a square and the other is not. If the second of these occurs, 
then say $\beta_2/\beta_1=d^2$, and since exactly one of
$\pm 2d$ is a square, then by replacing $d$ by $-d$ if necessary we can obtain $2d-c^2=0$.
Therefore if $p$ is inert and $p\equiv 3\bmod 4$, 
then  $\chi_4(-\beta_2/\beta_1)=1$ in all cases.
\end{proof}

  \begin{lma}\label{lma2adic-zero}
  Let $K=\bb{Q}(\zeta_3)$, let $\nu$ denote the valuation extending the 2-adic valuation, and take $\pi=2$ to be the uniformizer.
  Suppose that $a=0$, and let $b_1=2^{\nu(b_1)}\beta_1$, and $b_2=2^{\nu(b_2)}\beta_2$ where $\alpha,\beta_1,\beta_2\in K_\nu^\times$.
  Then $C(K_\nu)\neq\emptyset$ if and only if one of the following conditions holds.
  \begin{enumerate}
   \item $\nu(b_1)$ or $\nu(b_2)$ is even and for that $b_i$ the corresponding $\beta_i^3$ reduces to 1 or 5 $\bmod\,8A_\nu$.
         In particular if $b_1,b_2\in\bb{Z}_2$, then the condition on $\beta_i$ can be simplified to $\beta_i\equiv 1\bmod 4A_\nu$.
   \item $\nu(b_1)$ is even, $\nu(b_1)\equiv \nu(b_2)\bmod 4$, and $\beta_1x^4+\beta_2\equiv 2^k s\bmod 32A_\nu$ has a solution 
         for $k=0$, 2, 4, or 6 and some square $s$.
         In particular, if $b_1b_2\in\bb{Z}_2$, then such a solution exists if and only if one of the following is true:
         \begin{enumerate}
          \item $\beta_1,\beta_2$ are conjugates $\bmod\,8A_\nu$ and $(x-\beta_1^3)(x-\beta_2^3)$ is congruent $\bmod\,8A_\nu$ to one of the following
                polynomials
                \[
                 x^2+2x+1,\quad x^2+3,\quad x^2-2x+5,\quad x^2+4x+7;
                \]
                if $b_1,b_2\in\bb{Z}_2$, then this simplifies to $-\beta_1\equiv-\beta_2\equiv s\bmod 8A_\nu$ for some square $s$.
          \item $\beta_1+\beta_2$ is congruent to an integer $\bmod\,16A_\nu$, $\beta_1\beta_2$ is congruent to an integer $\bmod\,32A_\nu$,
                and $(x-\beta_1^3)(x-\beta_2^3)$ is congruent $\bmod\,8A_\nu$ to one of the following polynomials
                \[
                 x^2+3,\quad x^2+4x+3 ,\quad x^2+7,\quad x^2+4x+7;
                \]
                if $b_1,b_2\in\bb{Z}_2$, then this simplifies to $\beta_1+\beta_2\equiv 0\bmod 4A_\nu$.
          \end{enumerate}
     \item $\nu(b_1)$ is even, $\nu(b_1)\equiv \nu(b_2)+ 2\bmod 4$, and $x^2-(\beta_1+4\beta_2)\bmod 8A_\nu$ or $x^2-(4\beta_1+\beta_2)\bmod 8A_\nu$
           has a solution. In particular
         \begin{enumerate}
          \item if $b_1,b_2\in\bb{Z}_2$, then such a solution exists if and only if $\beta_1$ both $\beta_2$ are squares, or
          \item if $b_1,b_2\not\in\bb{Z}_2$ but $b_1b_2\in\bb{Z}_2$, then a solution exists if and only if 
               $(x-\beta_1^3)(x-\beta_2^3)$ is congruent mod $\bmod\,8A_\nu$ to one of the following polynomials
               \[
                x^2-2x+1,\quad x^2+4x+3,\quad x^2+2x+5,\quad x^2+7.
               \]
         \end{enumerate}
         
   \item $\nu(b_1)$ is odd, $\nu(b_1)\equiv \nu(b_2)\bmod 4$, and $\beta_1x^4+\beta_2\equiv 2^k s\bmod 32A_\nu$ has a solution 
         for $k=1$, 3, or 5, and some square $s$.
         In particular if $b_1b_2\in\bb{Z}_2$, then such a solution exists if and only if $\beta_1\beta_2$ is congruent
         to an integer $\bmod\,32A_\nu$ and $(x-\beta_1^3)(x-\beta_2^3)$ is congruent $\bmod\,8A_\nu$ to one of the following 
         polynomials
         \[
             x^2-2x+1,\quad x^2+3,\quad x^2-2x+5,\quad x^2+7;
         \]
         if $b_1,b_2\in\bb{Z}_2$, then this conditions simplifies to $\beta_1+\beta_2\equiv 0\text{ or }2\mod 8A_\nu$.
    \end{enumerate}
 \end{lma}
 
 \begin{rmk}
   Note that $b_1,b_2\in\bb{Z}_2$ and $\beta_1+\beta_2\equiv 0\bmod 4A_\nu$ implies that $\beta_1$ or $\beta_2\equiv 1\bmod 4A_\nu$,
 which is already covered in the first case.
 \end{rmk}

 \begin{proof}
The first case is as in Lemma~\ref{lma2adic-nonzero}. If the first case does not hold, then neither $b_1$ nor $b_2$ is a square in $K_\nu$,
which we will now assume. Let $u=2^{\nu(u)}\vep, w=2^{\nu(w)}\delta$ where $\vep,\delta\in A_\nu^\times$. We thus have
\[
  \nu(b_1u^4+b_2w^4)\geq\min(\nu(b_1)+4\nu(u),\nu(b_2)+4\nu(w)).
\]
If $C(K_\nu)\neq\emptyset$, then $b_1u^4+b_2w^4$ must be a square in $K_\nu$, and $|\nu(b_1)-\nu(b_2)+4(\nu(u)-\nu(w))|\geq 3$ 
implies that $b_1$ or $b_2$ must be a square giving a contradiction.
Let $\nu(b_2)-\nu(b_1)=4\mu+r$, where $r$ is a symmetric residue $\bmod\, 4$. Then we must have $\nu(u)-\nu(w)=\mu$ and in fact we may
assume that $\nu(w)=0$, since changing both $u$ and $w$ by a factor of $2$ changes $b_1u^4+b_2w^4$
by a factor of $2^4$, leaving the parity of $\nu(b_1u^4+b_2w^4)$ unchanged.
Thus we have reduced our considerations to the following cases
\[
 b_1u^4+b_2w^4=\begin{cases}
                2^{\nu(b_1)}(\beta_1\vep^4+\beta_2\delta^4 2^r) &\text{ if }r=0,1,2\\
                2^{\nu(b_1)-1}(\beta_1\vep^4 2+\beta_2\delta^4) &\text{ if }r=-1
               \end{cases}
\]
If $r=-1$ and $\nu(b_1)-1$ is even, then $\beta_1\vep^4 2+\beta_2\delta^4$ is in $A_\nu^\times$ and must be a square, 
which is not possible by \pref{eq:3}. On the other hand if $r=-1$ and $\nu(b_1)-1$ is odd, then 
$\nu(\beta_1\vep^4 2+\beta_2\delta^4)=0$,  so we cannot get an extra factor of $2$. If $\nu(b_1)$ is even and $r> 0$, 
then $\beta_1\vep^4+\beta_2\delta^4 2^r$ is in $A_\nu^\times$ and must be a square,
thus $r=1$ cannot occur by \pref{eq:3}, meaning that $\nu(b_2)$ is even also.
If $\nu(b_1)$ is odd and $r>0$, then we must have $\nu(\beta_1\vep^4+\beta_2\delta^4 2^r)>0$,
meaning that $r=0$, hence $\nu(b_1)\equiv\nu(b_2)\bmod 4$.
From these considerations it follows that $r=-1$ and $r=1$ never occur, while $r=2$ can only occur when $\nu(b_1)$ is
even, and $r=0$ can occur when $\nu(b_1)$ is even or odd. Hensel's lemma then allows the problem to be reduced to a
finite amount of computation, which Sage \cite{sagemath}.
 \end{proof}

 \begin{rmk}
   It is interesting to note that in the case where $\nu(b_1)$ and $\nu(b_2)$ are both even and congruent mod $4$.
   It is easy to construct a solution to $\beta_1\vep^4+\beta_2$ under the assumption
   $-\beta_1\equiv-\beta_2\equiv s\bmod 8A_\nu$ where $s$ is a square. Since $\zeta_3$ is a fourth power, we have
   \[
    \beta_1\zeta_3+\beta_2\equiv \beta_1(\zeta_3+1)\equiv -\zeta_3^2\beta_1\bmod 8A_\nu.
   \]
 \end{rmk}
 
\section{Bounds on rank}\label{SecRankBounds}

In this section we apply lemmas~\ref{lmapadic}~and~\ref{lma2adic-zero} to determine the
2-Selmer rank for elliptic curves of the type
\[
 E:y^2=x^3+bx,
\]
by applying the method of descent by 2-isogeny as described in Chapter X, Section 4 of \cite{Silverman};
see \pref{eq:selrk} above for the definition of the 2-Selmer rank. While the previous section
dealt with points on homogeneous spaces in a fairly general context, here we must be attentive
to the fact that it is the spaces
\[
 C_{b_1,b_2}:v^2=b_1u^4+b_2v^4
\]
with $b_1b_2=-4b$ that determine $S^{(\phi)}(E/K)$ while those with $b_1b_2=-b$ determine $S^{(\hat\phi)}(E'/K)$,
where
\[
 E':y^2=x^3-4bx.
\]
In particular, we note that $C_{b_1,b_2}$ as defined here, differs from Silverman's $C'_d$ or $C_d$
by a factor of $b_1=d$ after dehomogenizing with $u=1$. The point $(0,0)$ is a 2-torsion point on $E$ and on $E'$,
which corresponds with $b$ or $-4b$ in $K(S,2)$. Thus the curves $C_{b,1}$ and $C_{-4b,1}$, which always
have a solution, do not contribute to the 2-Selmer rank, since 2-torsion is removed (see the definition of $\SelRk_2(E/K)$
as given by \pref{eq:selrk} above).

For this section we will restrict out attention to imaginary quadratic number fields with class number 1 
such that 2 is inert, in other words $K=\bb{Q}(\sqrt{D})$ where $D=-3,-11,-19,-43,-67,$ or $-163$.
The class number 1 assumption is useful because it is easy to describe $K(S,2)$ in this case,
specifically it is generated by the irreducible elements dividing $b$ and $-4b$,
including units as factors up to squares. For imaginary quadratic number fields there is
only one archimedean place, and there is no trouble there since the completion of $K$ 
at that place is $\bb{C}$. The assumption that 2 is inert is useful mainly because
it yields an isomorphism between $K_\nu$ and $\bb{Q}_2(\zeta_3)$, thus all computer
computations with Sage~\cite{sagemath} can be done in the ring of integers of 
$K=\bb{Q}(\zeta_3)$ mod the ideal $8\Cal{O}_K$.

Several reductions can be made when considering pairs $(b_1,b_2)$. As noted above,
the pairs $(b,1)$ and $(-4b,1)$ can be ignored since they only account for the 2-torsion.
Since $K(S,2)$ is determined up to squares, we may assume that $(b_1,b_2)$ is square-free.
For $D\neq-3$, the only units are $\pm1$, and so if $b\in\bb{Z}$,
then we may assume that $b_1,b_2\in\bb{Z}$ (and thus in $\bb{Z}_p$ for all primes under consideration).
However, we can get away with this even in the case when $D=-3$, since the units $\zeta_3,\zeta_3^2$ 
are fourth powers in $K$, hence factors of $\zeta_3$ can be removed from $b_1,b_2$ by absorbing 
them in the variables $u$ and $w$. Finally, if $b$ is odd, then $C'_2$ in Silverman's notation
must be identified with $C_{8,2b}$, but so long as 2 is inert, Lemma~\ref{lma2adic-zero}
says that $C_{8,2b}$ has no solutions in $K_\nu$ where $\nu$ extends the 2-adic valuation.
It follows that if $b$ is odd then $2$ is not in $S^{\hat\phi}(E'/K)$ for any 
$K$ under consideration, and similarly with $-2$. We therefore consider only the remaining options 
in the tables of this section.

\begin{proof}[Proof of Theorem~\ref{thm-Inert}]
Let $b=\pm \prod_{i}^n p_i$, where each $p_i>2$ is inert,
$p_i\neq p_j$ when $i\neq j$, and let $b_i=\frac{b}{p_i}$ for all $i$.
Fix $r=1,3,5$, or $7$, and let $b\equiv r\bmod 8$.
Let $E_b:y^2=x^3+bx$ and $E'_b:y^2=x^3-4bx$. 
Clearly there is no problem over $\bb{C}$, and 
Lemma~\ref{lmapadic} applies to all odd primes under consideration here since if $5$ is inert then $|k_v|=25>7$.
If $\nu$ extends the $p_i$-adic valuation, since $p_i$ is inert we may assume that $b_1,b_2\in\bb{Z}$ and
that $\nu(b_1)=0$ or $\nu(b_2)=0$, hence $C(K_\nu)\neq\emptyset$.
It remains only to consider the valuation $\nu$ extending the 2-adic valuation. 
The following table summarizes the conditions on which $C_{b_1,b_2}(K_\nu)\neq\emptyset$ as indicated by Lemma~\ref{lma2adic-zero}.
\smallskip

\begin{tabular}{l|llllll}
 $(b_1,b_2)$&$(b_i,p_i)$&$(-b_i,-p_i)$&$(b_i,-4p_i)$&$(-b_i,4p_i)$&$(2b_i,-2p_i)$&$(-2b_i,2p_i)$\\
 \hline
$r=1$&True&True&True&True&True&True\\
$r=3$&True&True&$b_i\equiv 1\bmod 4$&$b_i\equiv 3\bmod 4$&$b_i\equiv 3\bmod 4$&$b_i\equiv 1\bmod4$\\
$r=5$&$b_i\equiv 1\bmod 4$&$b_i\equiv 3\bmod 4$&True&True&False&False\\
$r=7$&True& True&$b_i\equiv 1\bmod 4$&$b_i\equiv 3\bmod 4$&$b_i\equiv 1\bmod4$&$b_i\equiv 3\bmod 4$
\end{tabular}
\smallskip

This table makes it clear that the sizes of $S^{(\phi)}(E/K)$ and $S^{(\hat\phi)}(E'/K)$ do not
depend on the congruences satisfied by the primes $p_i$, but only on $b\equiv r\bmod 8$.
If $r=1$ then the size is maximal. If $r=3$ or $r=7$, then $-1$ is not in $S^{(\phi)}(E/K)$.
If $r=5$, then $2$ is not in $S^{(\phi)}(E/K)$ and $-1$ is not in $S^{(\hat\phi)}(E'/K)$.
Also $2$ is not in $S^{(\hat\phi)}(E'/K)$, since $b$ is odd as observed at the beginning of the section.
Since $\dim_{\bb{F}_2}\frac{E'(K)[\hat\phi]}{\phi(E(K)[2])}=1$ for all curves $E$
under consideration, the result follows immediately.
\end{proof}

In order to prove results involving split primes, it will be useful to have the following two lemmas.

\begin{lma}
 Let $K=\bb{Q}(\sqrt{D})$ where $D=-3,-11,-19,-43,-67,$ or $-163$. Let $p$ be a prime that splits over $K$,
 and let $p=\alpha\wbar{\alpha}$ be a factorization in which $\alpha\Cal{O}_K$ and $\wbar{\alpha}\Cal{O}_K$ are the prime ideals
 above $p$. Then, $(x-\alpha^3)(x-\wbar{\alpha}^3)\bmod 8A_\nu$ must be congruent to one of 
 \[
  x^2\pm 2 x+1,\, x^2+3,\, x^2\pm 2x+5,\, x^2+4x+7,
 \]
 and $(x-\alpha^3)(x+\wbar{\alpha}^3)\bmod 8A_\nu$ must be congruent to one of
 \[
  x^2+7,\, x^2+(4\zeta_3\pm 2)x+5,\, x^2+4x+3,\, x^2+(4\zeta_3\pm 2)x+1.
 \]
\end{lma}

\begin{proof}
If $\nu$ is the valuation extending the 2-adic valuation, then $K_\nu$ is isomorphic to $\bb{Q}_2(\zeta_3)$.
Let $\alpha=a+b\zeta_3$ in the completion. Since $\alpha^3\wbar{\alpha}^3=(\alpha\wbar{\alpha})^3=p^3\equiv p\bmod 8A_\nu$, then 
\[
 (x-\alpha^3)(x-\wbar{\alpha}^3)\equiv x^2-(2(a^3+b^3)-3ab(a+b))x+p\bmod 8A_\nu,
\]
and
\[
  (x-\alpha^3)(x+\wbar{\alpha}^3)\equiv x^2-3ab(a-b)(2\zeta+1)x-p\bmod 8A_\nu.
\]
Sage \cite{sagemath} can be used to check all cases $\bmod\,8A_\nu$.
\end{proof}

\begin{lma}
 Let $K=\bb{Q}(\sqrt{D})$ where $D=-3,-11,-19,-43,-67,$ or $-163$. Let $p$ be a prime that splits over $K$,
 and let $p=\alpha\wbar{\alpha}$ be a factorization in which $\alpha\Cal{O}_K$ and $\wbar{\alpha}\Cal{O}_K$ are the prime ideals
 above $p$. Then $\alpha$ is a square if and only if $\alpha^3$ satisfies the polynomial $x^2-2x+1\bmod 8A_\nu$.
\end{lma}

\begin{proof}
Without a computer it should be clear that $\alpha$ and $\wbar{\alpha}$, are either both squares or both non-squares,
from which it follows that $p\equiv 1\bmod 8$ is necessary. The remaining details are best left to a computer,
e.g. Sage \cite{sagemath}, and can be computed by methods similar to the previous lemma.
\end{proof}

\begin{proof}[Proof of Theorem~\ref{thm-split}]
 Let $K=\bb{Q}(\sqrt{D})$ where $D=-3,-11,-19,-43,-67,$ or $-163$, and let $p$ be a prime that splits over $K$. 

First, consider $E:y^2=x^3+px$ and $E':y^2=x^3-4px$.
Then $p$ can factor as $p=\alpha_1\alpha_2$, with $\alpha_1,\alpha_2$ belonging to different prime ideals,
which leads to a larger number of pairs $(b_1,b_2)$ then in the case where $p$ was inert.
For clarity, let $t=\alpha_1+\alpha_2$ be the trace in the field.
The tables below summarize the conditions on which $C_{b_1,b_2}(K_\nu)\neq \emptyset$, and can be obtained by applying 
lemmas~\ref{lma2adic-zero}~and~\ref{lmapadic}. 

\begin{tabular}{l|lp{3cm}}
 $(b_1,b_2)$&$(-1,-p)$&$(\pm \alpha_1,\pm \alpha_2)$\\
 \hline
  $\nu(2)$  & $p\not\equiv 5\bmod 8$&One if $p\equiv 1\bmod 4$, both otherwise.\\
  $\nu(\alpha_1)$& $p\equiv 1\bmod 4$&$\paren{\frac{\pm t}{p}}=1$\\
  $\nu(\alpha_2)$& $p\equiv 1\bmod 4$&$\paren{\frac{\pm t}{p}}=1$
\end{tabular}

\begin{tabular}{l|lllll}
 $(b_1,b_2)$&$(-1,4p)$&$(2,-2p)$&($-2,2p)$&$(\pm \alpha_1,\mp 4\alpha_2)$&$(\pm 2\alpha_1,\mp2\alpha_2)$\\
 \hline
  $\nu(2)$&$p\equiv 1\bmod 4$&$p\equiv 1,7\bmod 8$&$p\equiv 1,3\bmod 8$& Both if $p\equiv 1\bmod 4$& Both if $p\equiv 1\bmod 8$\\
  $\nu(\alpha_1)$&$p\equiv 1\bmod 4$&$p\equiv 1,7\bmod 8$&$p\equiv 1,3\bmod 8$&$\paren{\frac{\mp t}{p}}=1$&$\paren{\frac{\mp 2t}{p}}=1$\\
  $\nu(\alpha_2)$&$p\equiv 1\bmod 4$&$p\equiv 1,7\bmod 8$&$p\equiv 1,3\bmod 8$&$\paren{\frac{\pm t}{p}}=1$&$\paren{\frac{\pm 2t}{p}}=1$
\end{tabular}

 From these tables, we see that $S^{(\phi)}(E/K)$ always contains $-p$, contains $-1$ if and only if $p\equiv 1\bmod 4$, 
 contains $2$ or $-2$ if $p\not\equiv 5\bmod 8$,
 and contains $\alpha_1$ and $\alpha_2$ if $p\equiv 1\bmod 4$ and $\paren{\frac{t}{p}}=1$. As for $S^{(\hat\phi)}(E'/K)$, we never have $\pm 2$, but $-1$
 occurs when $p\equiv 1\bmod 8$, and $p$ always occurs. As for $\alpha_1$ and $\alpha_2$, the situation can be broken into two cases. If $p\equiv 1\bmod 4$,
 then we have both $\alpha_1$ and $\alpha_2$ or both $-\alpha_1$ and $-\alpha_2$, but not the other pair, if and only if $\paren{\frac{t}{p}}=1$.
 On the other hand if $p\equiv 3\bmod 4$, then $S^{(\hat\phi)}(E'/K)$ both $\alpha_1$ and $\alpha_2$ or both $-\alpha_1$ and $-\alpha_2$, but not the other pair,
 depending on which pair satisfies $\paren{\frac{t}{p}}=-1$.

Similarly for $E:y^2=x^3-px$ and $E':y^2=x^3+4px$ we obtain the following tables

\begin{tabular}{l|lll}
 $(b_1,b_2)$& $(1,-p)$ &$(-1,p)$&$(\pm \alpha_1,\mp \alpha_2)$\\
 \hline
  $\nu(2)$  & True &$p\not\equiv 3\bmod 8$& False\\
  $\nu(\alpha_1)$& True &$p\equiv 1\bmod 4$&$\paren{\frac{\mp t}{p}}=1$\\
  $\nu(\alpha_2)$& True &$p\equiv 1\bmod 4$&$\paren{\frac{\pm t}{p}}=1$
\end{tabular}

\begin{tabular}{l|lllp{3cm}p{3cm}}
 $(b_1,b_2)$&$(-1,-4p)$&$(2,2p)$&($-2,-2p)$&$(\pm \alpha_1,\pm 4\alpha_2)$&$(\pm 2\alpha_1,\pm 2\alpha_2)$\\
 \hline
  $\nu(2)$&$p\equiv 3\bmod 4$&$p\equiv 1,7\bmod 8$&$p\equiv 5,7\bmod 8$& One if $p\equiv 1\bmod 4$ & One if $p\equiv 1\bmod 4$, both if $p\equiv 3\bmod 8$\\
  $\nu(\alpha_1)$&$p\equiv 1\bmod 4$&$p\equiv 1,7\bmod 8$&$p\equiv 1,3\bmod 8$&$\paren{\frac{\pm t}{p}}=1$&$\paren{\frac{\pm 2t}{p}}=1$\\
  $\nu(\alpha_2)$&$p\equiv 1\bmod 4$&$p\equiv 1,7\bmod 8$&$p\equiv 1,3\bmod 8$&$\paren{\frac{\pm t}{p}}=1$&$\paren{\frac{\pm 2t}{p}}=1$
\end{tabular}

 From these tables, we see that $S^{(\phi)}(E/K)$, always contains $p$,
 never contains $-1$ or $-2$, but it contains $2$ if $p\equiv 1$ or 7 $\bmod 8$. The situation with
 $\alpha_1$ and $\alpha_2$ is complicated. If $p\equiv 7\bmod 8$, then none of $\pm \alpha_1,\pm \alpha_2,\pm 2\alpha_1,\pm 2\alpha_2$ are in $S^{(\phi)}(E/K)$.
 If $p\equiv 3\bmod 8$, then 
 \[
  \paren{\frac{2t}{p}}=1\iff \paren{\frac{t}{p}}=-1\quad\text{and}\quad \paren{\frac{-2t}{p}}=1\iff \paren{\frac{t}{p}}=1
 \]
so either both of $2\alpha_1$ and $2\alpha_2$ or both of $-2\alpha_1$ and $-2\alpha_2$ are in $S^{(\phi)}(E/K)$ but not the other pair. If $p\equiv 5\bmod 8$, then 
then both of $2\alpha_1$ and $2\alpha_2$ or  both of $-2\alpha_1$ and $-2\alpha_2$ are in $S^{(\phi)}(E/K)$ if and only if $\paren{\frac{t}{p}}=-1$, but not the other pair,
on the other hand from the column for $(\pm \alpha_1,\pm 4\alpha_2)$, we see that both of $\alpha_1$ and $\alpha_2$ or both of $-\alpha_1$ and $-\alpha_2$ are in
$S^{(\phi)}(E/K)$ if and only if $\paren{\frac{t}{p}}=1$. Thus $p\equiv 3\text{ or }5\bmod 8$, the dimension does not depend on $t$, even though
the particular elements in $S^{(\phi)}(E/K)$ do depend on $t$. If $p\equiv 1\bmod 8$, then $\alpha_1$ and $\alpha_2$ or $-\alpha_1$ and $-\alpha_2$ are generators 
if and only if $\paren{\frac{t}{p}}=1$.

As for $S^{(\hat\phi)}(E/K)$, $-1$ is a generator if and only if $p\equiv 1\bmod 4$, but $2$ and $-2$ never are since $p$ is odd.
Finally, $-p$ is always in $S^{(\hat\phi)}(E/K)$, and  $\alpha_1$ and $\alpha_2$ are in $S^{(\hat\phi)}(E'/K)$ if and only if $p\equiv 1\bmod 8$ and $\paren{\frac{t}{p}}=1$.
\end{proof}
 
\begin{rmk}
 The use of $\paren{\frac{t}{p}}=1$ is allowed because if $\chi_{\alpha_1}$ and $\chi_{\alpha_2}$ are 
 the characters of $\Cal{O}_K/(\alpha_1)$ and $\Cal{O}_K/(\alpha_2)$ respectively, extended by zero
 in the usual way, then 
\[
 \chi_{\alpha_1}(\alpha_2)=\chi_{\alpha_1}(t)\quad\text{and}\quad \chi_{\alpha_2}(\alpha_1)=\chi_{\alpha_2}(t).
\]
Then by the isomorphism $\bb{Z}/(p)\to\Cal{O}_K/(q_i)$ defined by $r+(p)\mapsto r+(q_i)$, it follows that
\[
 \paren{\frac{t}{p}}=\chi_{\alpha_1}(t)=\chi_{\alpha_2}(t).
\]
 \end{rmk}
 
 \begin{proof}[Proof of Theorem~\ref{thmCongNum}]
Let $n\in\bb{Z}^+$ be square free and divisible only by primes that are inert in $K$. Consider the curves
\[
 E_n:y^2=x^3-n^2x\quad\text{and}\quad E'_n: x^3+4n^2x
\]
The first has full two torsion over $K$, but the second does not. In particular we note that 
the points $(\pm n,0)$ on $E$ map to $4n^2$ and $\pm n$ in $K(S,2)$, thus $-1$ is always
in $S^{\hat\phi}(E'/K)$ but does not contribute to the rank.

The even and odd cases are dealt with separately. For each case we obtain a table
indicating when $C_{b_1,b_2}(K_\nu)\neq \emptyset$ by applying lemmas~\ref{lmapadic}~and~\ref{lma2adic-zero}.

Case 1: $n\equiv 1\bmod2$. For each $i$, let $n=n_ip_i$.
\smallskip

\begin{tabular}{l|llllll}
 $(b_1,b_2)$&$(\pm p_i,\mp p_in_i^2)$&$(p_i,4p_in_i^2)$&$(-p_i,-4p_in_i^2)$&$(2p_i,2p_in_i^2)$&$(-2p_i,-2p_in_i^2)$\\
 \hline
  $\nu|2$& True&$p_i\equiv 1\bmod 4$&$p_i\equiv 3\bmod 4$&$p_i\equiv 1\bmod 4$&$p_i\equiv 3\bmod 4$\\
  $\nu|p$& True&True& True& True& True
\end{tabular}
\smallskip

From this table we see that $S^{(\phi)}(E/K)$ contains 2, but not $-1$ so it contains exactly one of $\pm p_i$ for each $i$.
Since $n$ is odd, then $2$ is not in $S^{(\hat\phi)}(E/K)$ as we saw at the beginning of the section, but $-1$ is and
$\pm p_i$ is for all $i$.

Case 2: $n\equiv 0\bmod2$. Let $n=2m$ and for each $i$ let $m=m_ip_i$.
\smallskip

\begin{tabular}{l|lllll}
 $(b_1,b_2)$&$(\pm p_i,\mp 4p_im_i^2)$&$(\pm 2p_i,\mp 2p_im_i^2)$&$(p_i, 16p_im_i^2)$&$(-p_i,-16p_im_i^2)$&$(\pm2,\pm8m^2)$\\
 \hline
  $\nu|2$&True&True & True& True& False\\
  $\nu|p$&True&True & True& True& True
\end{tabular}
\smallskip

From this table we see that $S^{(\phi)}(E/K)$ does not contain 2, but it contains $-1$ and both of $\pm p_i$ for all $i$;
meanwhile $S^{(\hat\phi)}(E/K)$ contains everything. Unlike in the previous theorem, we have 
$\dim_{\bb{F}_2}\frac{E'(K)[\hat\phi]}{\phi(E(K)[2])}=0$. Note also that $2$ is counted among the prime factors of $n$
in the even case.
\end{proof}




\bibliography{MaximalRank-Kretschmer}{}
\bibliographystyle{plain}

\end{document}